\documentclass[10pt]{article}
\usepackage{latexsym}
\usepackage{ulem,rotating,amstext,amsmath,amssymb,amsthm,amscd,epic,eepic}
\usepackage[all]{xy}
\usepackage{graphicx}
\usepackage{color}
\usepackage{sectsty}
\usepackage{setspace}
\newtheorem{thm}{Theorem}[section]
\newtheorem{cor}[thm]{Corollary}
\newtheorem{lem}[thm]{Lemma}
\newtheorem{prop}[thm]{Proposition}
\newtheorem{defn}[thm]{Definition}

\newcommand{\quash}[1]{}			% Global Commenting

   \let\ud=\d       % under-bar, under-dot
\let\d=\delta

\def\sll#1{\rlap{\,\raise1pt\hbox{/}}{#1}}

\def\beq#1{\begin{equation}\label{#1}}
\def\eeq{\end{equation}}

\let\sss=\scriptscriptstyle
\let\SSS=\scriptstyle

\let\2=\underline

\let\3=\ldots
\def\rlx{\relax\leavevmode}
\def\inbar{\vrule height1.5ex width.8pt depth0pt}
\def\sinbar{\vrule height1ex width.45pt depth0pt}
\def\ssinbar{\vrule height.7ex width.35pt depth0pt}
\def\IC{{\relax\leavevmode
                \ifmmode\mathchoice
                       {\hbox{\kern.25em\inbar\kern-.3em{\sf C}}}
                       {\hbox{\kern.25em\inbar\kern-.3em{\sf C}}}
                       {\hbox{\kern.20em\sinbar\kern-.25em$\SSS\sf C}}
                       {\hbox{\kern.18em\ssinbar\kern-.22em$\sss\sf C }}
                \else{\hbox{\kern.3em\inbar\kern-.3em{\rm C}}}\fi}}
\def\Ik{{\rlx{\rm I\kern-.18em k}}}

\def\IP{{\rlx{\rm I\kern-.18em P}}}
\def\IQ{{\relax\leavevmode
                \ifmmode\mathchoice
                       {\hbox{\kern.33em\inbar\kern-.3em{\rm Q}}}
                       {\hbox{\kern.33em\inbar\kern-.3em{\rm Q}}}
                       {\hbox{\kern.28em\sinbar\kern-.25em$\SSS\rm Q}}
                       {\hbox{\kern.25em\ssinbar\kern-.22em\sss\rm Q}}
                \else{\hbox{\kern.3em\inbar\kern-.3em{\rm C}}}\fi}}
\def\IR{{\rlx{\rm I\kern-.18em R}}}
\def\ZZ{{\relax\leavevmode
                \ifmmode\mathchoice
                       {\hbox{\sf Z\kern-.4em Z}}
                       {\hbox{\sf Z\kern-.4em Z}}
                       {\lower.9pt\hbox{\scriptsize\sf Z\kern-.36em Z}}
                       {\lower1.2pt\hbox{\tiny\sf Z\kern-.36em Z}}
                \else{\sf Z\kern-.4em Z}\fi}}
\def\Ione{{\rlx{\rm 1\kern-3pt l}}}
\def\JGP#1{{\itshape J.\,Geom.\,Phys.\,}{\bf#1\,}}

\def\AMS#1{{\itshape Ann.\,Math.\,Stud.\,}{\bf#1\,}}

\def\IM#1{{\itshape Inv.\,Math.\,}{\bf#1\,}}

\def\NP#1{{\itshape Nucl.\,Phys.\,}{\bf#1\,}}

\def\MPL#1{{\itshape Mod.\,Phys.\,Lett.\,}{\bf#1\,}}
\def\PRL#1{{\itshape Phys.\,Rev.\,Lett.\,}{\bf#1\,}}
\def\CMP#1{{\itshape Commun.\,Math.\,Phys.\,}{\bf#1\,}}
\def\CP#1#2{\relax\ifmmode\IP^{#1}_{#2}\else\IP$^{#1}_{#2}\fi}

\def\ytem#1{\par\noindent\hbox to\parindent{\hss\em#1.~}\ignorespaces}

\thicklines     
\def\Label#1{\label{#1}%
    \smash{\hbox to0pt{\raise1ex\hbox{\tiny\{#1\}}\hss}}}
\def\noLabels{\let\Label=\label}

\catcode`@=11   % WATCH OUT !!!!! This must be reset in the end !!!!!
\newbox\t@b@x
\def\rightarrowfill{$\m@th \mathord- \mkern-6mu
      \cleaders\hbox{$\mkern-2mu \mathord- \mkern-2mu$}\hfill
       \mkern-6mu \mathord\rightarrow$}
\def\tooo#1{\setbox\t@b@x=\hbox{\footnotesize$#1$}%
              \mathrel{\mathop{\hbox to\wd\t@b@x{\rightarrowfill}}%
               \limits^{#1}}\,}
\def\leftarrowfill{$\m@th \mathord\leftarrow \mkern-6mu
      \cleaders\hbox{$\mkern-2mu \mathord- \mkern-2mu$}\hfill
       \mkern-6mu \mathord-$}
\def\froo#1{\setbox\t@b@x=\hbox{\footnotesize$#1$}%
              \mathrel{\mathop{\hbox to\wd\t@b@x{\leftarrowfill}}%
               \limits^{#1}}\,}
\catcode`@=12                   % You see ?

\numberwithin{equation}{section}
\begin{document}
%\leftline{math.AG/0210394}

\begin{center}

{\LARGE\bf A Perverse Sheaf Approach Toward a \\[2mm]
            Cohomology Theory for String Theory}\\[10mm]

{\bf Abdul ~Ra\ud{h}m\={a}n}\footnote{arahman@howard.edu}\\[1mm]
{\it  Department of Physics and Astronomy\\
      Howard University\\
      Washington, DC 20059}\\[5mm]

{\bf ABSTRACT}\\[3mm]
\parbox{5in}{
We present the construction and properties of a self-dual perverse sheaf $\uuline{\mathcal{S}_0}^\cdot$
whose cohomology fulfills some of the requirements of String theory as outlined in ~\cite{hub:stringy}.  
The construction of this $\uuline{\mathcal{S}_0}^\cdot$ utilizes techniques that follow from MacPherson-Vilonen ~\cite{mv:elem}.  Finally, we will discuss its properties as they relate to String theory.
}
\end{center}

\section{Introduction}

In analyzing ways to move between Calabi-Yau manifolds, Green and Hubsch in 
~\cite{greenhub:conifold1},~\cite{greenhub:conifold2} showed 
there are cases within String theory
that admit mildly singular target spaces.  These mildly singular spaces, termed conifolds 
~\cite{greenhub:conifold1},~\cite{stro:massless},~\cite{hubrah:geomhom}
consist of the usual smooth target spaces with the addition of zero-dimensional
singularities that stratify these spaces.  Since these spaces are singular,
the usual techniques for calculating cohomology do not apply.  Ideally, a cohomology
theory for String theory would apply to both smooth and 
not `too' singular target spaces\footnote{The only known criterion on the 
severity of singularization comes from supersymmetry, which is essential for vacuum 
stability~\cite{hub:singspace}.} so that determination of stringy vacua could be possible 
in all cases.  In an effort to
qualify the mathematical requirements for this (co)homology theory,
Hubsch proposed in~\cite{hubsch:best} a working definition for 
a homology theory for String theory in the case of a mildly degenerate target space.  This is recalled in
the following definition.
\begin{defn}
Let Y be a 2n-dimensional stratified space with a single isolated singularity, y. Then
\begin{equation} \label{qualphys}
SH_k(Y)=
\begin{cases}
H_k(Y), & \text{ $k > n$;}\\
H_n(Y-y) \cup H_n(Y), & \text{ $k$;}\\
H_k(Y-y), & \text{$k < n$}
\end{cases}
\end{equation}
\end{defn}
\noindent
The middle dimension case, $k$, $H_n(Y-y) \cup H_n(Y)$ is a qualitative 
way of expressing the String theory requirement that the homology group 
contain cycles from both $H_n(Y-y)$ and $H_n(Y)$.  
This means that the middle dimension homology group should be larger than
either $H_n(Y-y)$  or $H_n(Y)$.  Many homology theories were examined to attempt to
fulfill this requirement but only intersection
homology\footnote{A detailed discussion can be found in~\cite{gormac:ih1} and ~\cite{gormac:ih2}.}  
 seemed to possess most of the requirements.  However, in the middle dimension, intersection
homology provided fewer (co)cyclic classes than required by String theory
\cite{stro:massless}.  It was concluded that $H^k(Y;\uuline{IC}^\cdot)$ seemed 
to be the correct choice in all degrees except $k$, where $\uuline{IC}^\cdot$
is the sheaf of intersection chains defined on p.77 of reference~\cite{gormac:ih2}.
In this paper we will discuss the construction of a complex of sheaves $\uuline{\mathcal{S}_0}^\cdot$ such that 
$H^k(Y;\uuline{\mathcal{S}_0}^\cdot)=H^k(Y;\uuline{IC}^\cdot)$ in all degrees $k \neq n$,
but will yield more cohomology in the  
middle dimension as predicted in String Theory~\cite{stro:massless}.
\subsection{K\"ahler Package}
As mentioned in the previous section, massless fields in superstring compactifications were identified
with cohomology classes on the target space.   However, a troubling consequence 
occurs when the target space is singular.  Essentially, different
cohomology theories on singular target spaces yield different results thereby making it difficult to determine
which theory physics may favor.  Several important characteristics of the cohomology, 
which correspond to the massless fields, are based on general properties of field theories, 
specifically, the (2,2)-supersymmetric 2-dimensional world-sheet field theories.  These properties, known
as the K\"ahler package, should hold for singular and smooth target spaces.  
 Let $\mathcal{Y}$ be a smooth Calabi-Yau target space.  We will use refs. ~\cite{hubsch:best} and ~\cite{cheegormac:L2cohom} for the definition of the K\"ahler package which is stated as follows.  
\begin{enumerate}
\item
Hodge Decomposition: $H^r(\mathcal{Y},\mathbb{C})=\bigoplus_{p+q=r}H^{p,q}(\mathcal{Y})$.
\item
Complex Conjugation: $H^{(p,q)}(\mathcal{Y})=\overline{H^{(q,p)}(\mathcal{Y})}$, which follows
from CPT conjugation in the world-sheet field theory,
\item
Poincare Duality: $H^{n-p,n-q}(\mathcal{Y}) \times H^{p,q}(\mathcal{Y}) \to H^{n,n}(\mathcal{Y})$ is non-degenerate.
\item
Kunneth Formula: Given two topological spaces $X$ and $Y$. 
Then $H^r(X \times Y)=\bigoplus_{p+q=r}H^p(X) \otimes H^q(Y)$ which says that in a 
product of two spaces, harmonic forms are products
of a harmonic form from one space and a harmonic form from the other space, with 
their degrees added.
\item
Lefshetz $SL(2,\mathbb{C})$ action 
%(with $\omega$ the K\"ahler form on an n dimensional manifold X)
%Let $\eta \in H^{*,*}(X)$.  Define $L(\eta) = \omega \wedge \eta$, $\Lambda=(-)^{-n}*L*$, and $h=[L,\Lambda]$.\\
%Then the following relations hold: (1) $[h,L]=L$ and $[h,\Lambda] = -\Lambda$ and (2) the induced decomposition 
%from $[h,L]=L$ and $[h,\Lambda] = -\Lambda$ into irreducible $SL(2,\mathbb{C})$ representations (Hard Lefshetz Theorem).

\end{enumerate}
(Remark: The Lefshetz $SL(2,\mathbb{C})$ action does not have any recognizable counterpoint in the physics applications, so we ignore it for now.)  \\
 \indent
 In the case that $\mathcal{Y}$ is a singular Calabi-Yau target space, String theory suggests that these properties of the K\"ahler package should be preserved.  Hence, a cohomology theory for String theory should not only meet the cohomology requirements qualitatively outlined by Hubsch in~\cite{hub:stringy}, but should also meet the properties of the K\"ahler package whether $\mathcal{Y}$ is singular or smooth.
\\
\indent
In this paper, we show there exists a perverse sheaf $\uuline{\mathcal{S}_0}^\cdot$ 
that fulfills these cohomology requirements 
and satisfies one part of the K\"ahler Package.  Proving the remaining parts are currently open problems.  In the next
section we present the mathematical tools used in this paper.

\subsection{Useful Mathematical Tools} \label{conv}
In this paper we will use the following definitions and statements.
\begin{defn} \label{strat}
A 2n-dimensional simple stratified space $Y$  is a compact topological space with 
one ``singular" point $y \in Y$ such that:
\begin{enumerate}
\item
$Y-\{y\}$ is a smooth 2n-dimensional manifold called the ``non-singular" part of $Y$,
\item
There exists a neighborhood $U_y$ of the singular point $y$ whose closure 
is homeomorphic to the cone $cL$ over $L$, where $L$ is a
compact $(n-1)$-dimensional submanifold of $Y$, by a homeomorphism 
$\phi:cL \to U_y$ such that $\phi(*)=y$ where * represents the cone point. The
open cone over the link will be denoted by $c^oL$.
\end{enumerate} 
\end{defn}
Let $Y$ be a simple stratified space with $Y^o$ the non-singular part of $Y$.  Define the inclusions
$j:Y^o \hookrightarrow Y$, and for the singular point $y \in Y$ 
define the inclusion $i:\{y\} \hookrightarrow Y$.
We will denote the sections of a complex of  sheaves $\uuline{S}$ over an open set U as 
as either $\Gamma(U,\uuline{S})$ or $\uuline{S}(U)$.  A constructible complex of sheaves will
be denoted as $\uuline{K}^\cdot$.  The sheaves and complexes of sheaves used in this
paper will be constructible.  In addition, $H^*_c$ represents cohomology with compact supports.\\
\indent
A morphism $f:\uuline{K}^\cdot \to \uuline{L}^\cdot$ of complexes of sheaves is said 
to be a quasi-isomorphism if  
$\uuline{H}^n(f):\uuline{H}^n(\uuline{K}^\cdot) \to \uuline{H}^n(\uuline{L}^\cdot)$ 
is an isomorphism for any n.  
Let $\uuline{K}^\cdot$ and $\uuline{L}^\cdot$ be quasi-isomorphic
complexes of sheaves in an abelian category $\mathbb{A}$. The derived category of $\mathbb{A}$, 
$\mathcal D(\mathbb{A})$, is
a category in which these two complexes of sheaves are  
isomorphic. The bounded below derived category, $\mathcal{D}^b_Y(\mathbb{A})$ , of the abelian category 
$\mathbb{A}$ on $Y$ has objects, complexes of sheaves, that are zero below some degree.
Discussions about the derived category can be found in 
sect.  1.7 in ~\cite{kassch:shvmfd}, sect. 1.3 in ~\cite{dimca:shv},
or sect. 4.1 in ~\cite{gelfman:homalg}.  \\
\indent
In this paper, $\mathcal{D}^b_Y$  will denote the bounded below derived category of 
constructible complexes sheaves of $\mathbb{Q}$-vector spaces 
on $Y$.  It is known that $\mathcal{D}^b_Y$ has enough injective objects in the sense of 
sect. 1.3 in Ref.~\cite{dimca:shv}. Furthermore, we will restrict our treatment to the case of all local systems being the constant sheaf $\uuline{\mathbb{Q}}$ with $Y$ a simple stratified space.  \\
\indent
A functor T from the category of sheaves to an abelian category $\mathbb{A}$
gives rise to a functor RT from the derived category of sheaves to the 
derived category of $\mathbb{A}$.  We will be using the definitions of the pullback, pushforward, extension by 
zero, and Verdier duality functors from ~\cite{gelfman:homalg}.  
Let $\uuline{A} \in \mathcal{D}^b_Y$ be a complex of sheaves
on $Y^o$.  Its pushforward by $j$ is defined to be the complex of sheaves $Rj_* \uuline{A}$ on $Y$
whose sections over an open set $U \subset Y$ are defined as
$Rj_*(\uuline{A})(U) = \uuline{A}(j^{-1}(U))$ and the restriction
from $U$ to $V \subset U$ is induced by the restriction from $j^{-1}(U)$ to
$j^{-1}(V)$ (p. 45 in ~\cite{gelfman:homalg}).  Let $\uuline{B} \in \mathcal{D}^b_Y$ 
be a complex of sheaves on $Y$.  
Its pullback, $j^* \uuline{B}$, is a sheaf on $Y^o$ such that 
there is an isomorphism $Hom(j^* \uuline{B}, \uuline{A}) \simeq Hom(\uuline{B}, j_* \uuline{A})$.
(p. 46 in ~\cite{gelfman:homalg})

Let $f:U \to Y$ be an open subset of Y, let $\uuline{\mathcal{F}} \in \mathcal{D}^b_Y$ 
be a complex of sheaves on U.  The extension by zero complex of sheaves 
$Rj_! \uuline{\mathcal{F}}$ denotes the subsheaf of $Rj_* \uuline{\mathcal{F}}$ given by
$\Gamma(V,Rj_! \uuline{\mathcal{F}})=\{s \in \Gamma(U \cap V,\uuline{\mathcal{F}})|supp(s) 
\text{ is closed relative to V}\}$ where V is an open subset of U and $supp(s)$ denotes the
support of the section s. More details can be found in Chap. 2 and 5 in ~\cite{gelfman:homalg}.  \\
\indent
We will need the following facts about Verdier duality.  Let $Y$ and $Z$ be topological
spaces and let $f:Y \to Z$ be a map. Verdier defined a map $f^!:\mathcal{D}^b_Z \to \mathcal{D}^b_Y$.  
The Verdier duality theorem is a canonical isomorphism in $\mathcal{D}^b_Z$,
\begin{equation} \label{verddualthm}
Rf_*R\uuline{Hom}^\cdot(\uuline{A}^\cdot,f^! \uuline{B}^\cdot) \simeq 
R\uuline{Hom}^\cdot(Rf_! \uuline{A}^\cdot,\uuline{B}^\cdot)
\end{equation}
for any $\uuline{A}^\cdot \in \mathcal{D}^b_Y$ and $\uuline{B}^\cdot \in \mathcal{D}^b_Z$.

\begin{defn}{(Dualizing Sheaf)}\label{dualshf}
Let $Y$ be a locally compact topological space and let $f:Y \to \{pt\}$ be the map to a point. 
Verdier defined the dualizing sheaf to be $\mathbb{D}_Y := f^!\mathbb{Q}$.  
\end{defn}
Let Y be a locally compact topological space and let
$\uuline{A}^\cdot \in \mathcal{D}^b_Y$ be a complex of sheaves.  
In ref. ~\cite{verd:dualcoh}, Verdier defined the duality functor $\mathcal{D}_V$ by 
$\mathcal{D}_V(\uuline{A}^\cdot) := \uuline{RHom}(\uuline{A}^\cdot, \mathbb{D}_Y)$. 
\begin{defn}{(Verdier Duality Functor)} \label{verdduualfunc}
The duality functor $\mathcal{D}_V$ has the following properties (p. 92, \cite{gormac:ih2}).
Let $Y$ and $Z$ be topological spaces and let $f:Y \to Z$ be a map.  
There exist isomorphisms in $\mathcal{D}^b_Y$ as follows:
\begin{enumerate}
\item
$\mathbb{D}_Y \cong \mathcal{D}_V(\mathbb{Q}) \cong f^!(\mathbb{D}_Z)$
\item
$\uuline{A}^\cdot \cong \mathcal{D}_V(\mathcal{D}_V(\uuline{A}^\cdot))$
\end{enumerate}
\end{defn}
\begin{defn}{(Verdier Dual Pairing ~\cite{verd:dualcoh})}\label{vdp}
Let $\uuline{S}^\cdot$ and $\uuline{T}^\cdot$ be complexes of sheaves defined on the simple stratified space $Y$.
A Verdier dual pairing between $\uuline{S}^\cdot$ and $\uuline{T}^\cdot$ is a morphism 
$\phi:\uuline{S}^\cdot \otimes \uuline{T}^\cdot \to \mathbb{D}_Y[-2n]$ which induces a 
quasi-isomorphism $\uuline{S}^\cdot \to \uuline{RHom}^\cdot( \uuline{T}^\cdot,\mathbb{D}_Y[-2n])$ where 
$\mathbb{D}_Y$ is the dualizing sheaf on $Y$.
\end{defn}

\begin{prop} \label{indmap}
Suppose $\uuline{S}^\cdot$ and $\uuline{T}^\cdot$  are complexes of sheaves with a Verdier dual pairing
$\phi$ on the simple stratified space $Y$.  Then the morphism $\phi$ induces a non-degenerate pairing on cohomology
$\hat{\phi}: H^i(Y,\uuline{S}^\cdot) \otimes H^{2n-i}(Y,\uuline{T}^\cdot) \to H^0(Y,\mathbb{D}_Y) \cong \uuline{\mathbb Q}$.
\end{prop}
Two particular distinguished triangles will be used in this paper.
Let U be an open set in a topological space Y and let $Z$ be its closed complement.  
Let $\uuline{S}^\cdot$ be a complex of sheaves on $Y$.  
Then there are distinguished triangles defined as follows.
\begin{equation}\label{d1} 
\xymatrix{
Ri_*i^!\uuline{S}^\cdot \ar[rr]& &  \uuline{S}^\cdot \ar[dl]\\ 
& Rj_*j^*\uuline{S}^\cdot \ar[ul]^{[1]}\\ }
\end{equation}
\vspace{1cm}
\begin{equation} \label{d2}
\xymatrix{
Rj_!j^*\uuline{S}^\cdot \ar[rr]& &  \uuline{S}^\cdot \ar[dl]\\ 
& Ri_*i^*\uuline{S}^\cdot \ar[ul]^{[1]}\\ }
\end{equation}
These distinguished triangles (\ref{d1}) and (\ref{d2}) induce long exact sequences on cohomology,  
\begin{equation} \label{d1coh}
... \to H^n(Y;Ri_*i^!\uuline{S}^\cdot) \to H^n(Y;\uuline{S}^\cdot) \to H^n(Y;Rj_*j^*\uuline{S}^\cdot) \to H^{n+1}(Y;Ri_*i^!\uuline{S}^\cdot) \to ...
\end{equation}
and
\begin{equation} \label{d2coh}
... \to H^n(Y;Rj_!j^*\uuline{S}^\cdot) \to H^n(Y;\uuline{S}^\cdot) \to H^n(Y;Ri_*i^*\uuline{S}^\cdot) \to H^{n+1}(Y;Rj_!j^*\uuline{S}^\cdot) \to ...
\end{equation}
respectively.  We will require another relation that arises through manipulation of 
distinguished triangles.
\begin{lem}\label{jdualident}
For any complex of sheaves $\uuline{\mathcal{P}}^\cdot$ on the simple stratified space $Y$, 
there is a natural isomorphism $H^k(i^*Rj_*\uuline{\mathcal{P}}^\cdot) \cong H^{k+1}(i^!Rj_!\uuline{\mathcal{P}}^\cdot)$.
\end{lem}
\begin{proof}
Recall that for any complex of sheaves $\uuline{\mathcal{S}}^\cdot$ we have the distinguished triangle (\ref{d1}),
\begin{equation}
\xymatrix{
Ri_*i^!\uuline{\mathcal{S}}^\cdot \ar[rr]& &  \uuline{\mathcal{S}}^\cdot \ar[dl]\\ 
& Rj_*j^*\uuline{\mathcal{S}}^\cdot\ar[ul]^{[1]}\\ }
\end{equation}
Let $\uuline{\mathcal{S}}^\cdot = Rj_!\uuline{\mathcal{P}}^\cdot$ where $\uuline{\mathcal{P}}^\cdot$ is
a complex of sheaves.  Then we have,
\begin{equation} \label{modtri1}
\xymatrix{
Ri_*i^!Rj_!\uuline{\mathcal{P}}^\cdot \ar[rr]& &  Rj_!\uuline{\mathcal{P}}^\cdot \ar[dl]\\ 
& Rj_*j^*Rj_!\uuline{\mathcal{P}}^\cdot\ar[ul]^{[1]}\\ }
\end{equation}
We know that $Rj_*j^*Rj_!\uuline{\mathcal{P}}^\cdot = Rj_*\uuline{\mathcal{P}}^\cdot$.  Substituting into
eqn. (\ref{modtri1}) above we have
\begin{equation} \label{modtri2}
\xymatrix{
Ri_*i^!Rj_!\uuline{\mathcal{P}}^\cdot \ar[rr]& &  Rj_!\uuline{\mathcal{P}}^\cdot \ar[dl]\\ 
& Rj_*\uuline{\mathcal{P}}^\cdot\ar[ul]^{[1]}\\ }
\end{equation}
Applying $i^*$ to eqn. (\ref{modtri2}) we obtain,
\begin{equation} \label{modtri3}
\xymatrix{
i^*Ri_*i^!Rj_!\uuline{\mathcal{P}}^\cdot \ar[rr]& &  i^*Rj_!\uuline{\mathcal{P}}^\cdot \ar[dl]\\ 
& i^*Rj_*\uuline{\mathcal{P}}^\cdot\ar[ul]^{[1]}\\ }
\end{equation}
This simplifies to, 
\begin{equation} \label{modtri4}
\xymatrix{
i^!Rj_!\uuline{\mathcal{P}}^\cdot \ar[rr]& &  i^*Rj_!\uuline{\mathcal{P}}^\cdot \ar[dl]\\ 
& i^*Rj_*\uuline{\mathcal{P}}^\cdot\ar[ul]^{[1]}\\ }
\end{equation}
Since $Rj_!$ is the extension by zero functor, it follows that $i^*Rj_!\uuline{\mathcal{P}}^\cdot = 0$. Hence,
\begin{equation} \label{modtri5}
\xymatrix{
i^*Rj_*\uuline{\mathcal{P}}^\cdot \ar[r]^{[1]}  & i^!Rj_!\uuline{\mathcal{P}}^\cdot\\
}
\end{equation}
is an isomorphism in all degrees $k > 0$.  Applying $H^*$ to eqn. (\ref{modtri5}) we get,
\begin{equation} \label{jdlidval}
H^{k}(i^*Rj_*\uuline{\mathcal{P}}^\cdot) \cong H^{k+1} (i^!Rj_!\uuline{\mathcal{P}}^\cdot)
\end{equation}
for all $k >0$.
\end{proof}
\begin{lem} \label{cohvalues}
Let $Y$ be a simple stratified space and let $\uuline{\mathcal{P}}^\cdot$ be a complex of sheaves on Y.  
Let $y \in Y$ be the singular point and let $U_{y}$ be a distinguished neighborhood of $y$ so that
$\bar{U}_{y} \simeq cL_{y}$ and $U_{y} \simeq c^oL_{y}$.
Then there are natural isomorphisms
\begin{enumerate}
\item
$H^m(Y;i_*i^*\uuline{\mathcal{P}}^\cdot) \cong H^m(Y^o;\uuline{\mathcal{P}}^\cdot)$, $\forall $ $m > 0$,
\item
$H^m(Y;i_!i^*\uuline{\mathcal{P}}^\cdot) \cong H^m_c(Y^o;\uuline{\mathcal{P}}^\cdot)$, $\forall $ $m > 0$,
\item
$H^m(Y;j_*j^*\uuline{\mathcal{P}}^\cdot) \cong H^m(U_y;\uuline{\mathcal{P}}^\cdot)$, $\forall $ $m > 0$,
\item
$H^m(Y;j_*j^!\uuline{\mathcal{P}}^\cdot) \cong H^m_c(U_y;\uuline{\mathcal{P}}^\cdot)$, $\forall $ $m > 0$,
\item 
$H^m(Y;j^*i_*\uuline{\mathcal{P}}^\cdot) \cong H^m(L \times (0,1);\uuline{\mathcal{P}}^\cdot)$, $\forall $ 
$m > 0$ where $L$ is the link of the singular point $y$.
\end{enumerate}
\end{lem}
\noindent
Remark: The proof of the results in Lemma \ref{cohvalues} follow in a straight forward manner from sheaf theory.
For the remainder of this paper, $L$ and $L_y$ will be used interchangeably to denote
the link of the point $y \in Y$. 
The derived category, quasi-isomorphisms, right derived functors and related topics are taken from the references
~\cite{dimca:shv}, ~\cite{gelfman:homalg}, and ~\cite{iver:cohsheav}.
In the next section we present the motivation for constructing the perverse sheaf $\uuline{\mathcal{S}_0}^\cdot$.
\subsection{Mathematical Approach}
The objective of this paper is to construct the minimal object $\uuline{\mathcal{S}_0}^\cdot$ using the 
technique of MacPherson and Vilonen (Thm. 2.1 in ref. ~\cite{mv:elem}), show it provides the necessary cohomology in all degrees, discuss its properties, and then provide some qualitative insight to applications to String theory.
The cohomology requirement for the middle dimension as stated in the introduction can be expressed mathematically as follows.  Let Y be a simple stratified space.
We seek a minimal object $\uuline{\mathcal{S}_0}^\cdot$ such that $H^n(Y;\uuline{\mathcal{S}_0}^\cdot)$ fits 
into the following diagram.  
\begin{equation} \label{basicphys}
\xymatrix{
& H^n(Y;\uuline{\mathcal{S}_0}^\cdot) \ar@{->>}[dr]^d \\
H^n(Y;\uuline{\mathbb{Q}}) \ar@{^{(}->}[ur]^c  \ar[rr]^\gamma  & & H^{n}(Y - y ;\uuline{\mathbb{Q}})}
\end{equation}
where $c$ is an injection, $d$ is a surjection, and $\gamma$ is the restriction map on cohomology.
This minimal object $\uuline{\mathcal{S}_0}^\cdot$ would have cohomology greater than either $H_n(Y-y)$ or $H_n(Y)$ for $k$,
but $H^k(Y;\uuline{\mathcal{S}_0}^\cdot) \cong H^k(Y;\uuline{IC}^\cdot)$ in all other degrees.
The existence and construction of such an $\uuline{\mathcal{S}_0}^\cdot$ that yields the desired cohomology
in the case of a singular isolated point is the subject of this paper.  
In order to construct $\uuline{\mathcal{S}_0}^\cdot$ on $Y$, we will use the method developed by 
MacPherson and Vilonen presented in~\cite{mv:elem}.
In this technique, a perverse sheaf will be
constructed from datum on the non-singular part, $Y^o$ of Y.  Although~\cite{mv:elem} describes 
how to construct perverse sheaves on stratified spaces with higher dimensional strata and 
non-trivial local systems defined on $Y^o$, we have adapted this technique for 
the case of a space $Y$ that has a 'simple stratification' made up of two parts: the singular point $y \in Y$
and the smooth $Y^o$.  In addition, the constant local system defined on $Y^o$ will be used in all cases. 
In the next section, the category of perverse sheaves and the zig-zag category will be presented.
\begin{section}{Perverse Sheaves and the Zig-Zag Category}
In this section we will present a condensed discussion of the category of perverse sheaves and the Zig-zag category.
Throughout this paper we restrict ourselves to conventions as presented in section \ref{conv}.
%
% -------------------------------------------------------------
%	  				Perverse Sheaves 
% -------------------------------------------------------------
%
\begin{subsection}{The Category of Perverse Sheaves} \label{catperv}
\begin{defn} \label{defnpervshf}
The category of perverse sheaves $\mathbb{P}(Y)$ is the full sub-category of  $\mathcal{D}^b(Y)$ whose 
objects are complexes of sheaves $\uuline{\mathcal{S}}^\cdot$ which 
satisfy the following properties:
\begin{enumerate}
\item
There exists $M>0 \in \mathbb{Z}$ such that $\uuline{\mathcal{S}}^i=0$ $\forall$ $i < M$ (bounded below)
\item
The complex of sheaves $j^*\uuline{\mathcal{S}}^\cdot$ is quasi-isomorphic to a local system on $Y^o$ (in degree 0).
In other words,
\begin{enumerate}
\item
$\uuline{H}^k(j^*\uuline{\mathcal{S}}^\cdot) =0$ if $k \neq 0$
\item
$\uuline{H}^0(j^*\uuline{\mathcal{S}}^\cdot)$ is a local system
\end{enumerate}
\item \label{spt}
$H^k(i^*\uuline{\mathcal{S}}^\cdot)=0 \text{ for } k>n \text{ (support)}$ %\label{spt}
\item \label{cspt}
$H^k(i^!\uuline{\mathcal{S}}^\cdot)=0 \text{ for } k<n \text{ (cosupport)}$ %\label{cspt}
\end{enumerate}
\end{defn}
\noindent
Remark:
We will be primarily interested in the sub-category $\mathbb{P}_\mathbb{Q}(Y)$ of $\mathbb{P}(Y)$ 
that consists of perverse sheaves $\uuline{\mathcal{S}^\cdot}$ where the local system $\uuline{H}^0(j^*\uuline{\mathcal{S}}^\cdot)$
is the constant sheaf $\uuline{\mathbb{Q}}$.  Let $\mathbb{P}(Y^o)$ represent the category of perverse 
sheaves on $Y^o$. This is equivalent to the category of local systems defined on $Y^o$.
\end{subsection} 
%
% ---------------------------------------------
%				Zig-Zag Category 
% ---------------------------------------------
%
\begin{subsection}{Zig-Zag Category}
\begin{subsubsection}{Definitions and Properties}
We have modified the definition of the Zig-Zag category $Z(Y,y)$ as taken from  p. 409 in Ref.~\cite{mv:elem}.
\begin{defn}{($Z(Y,y)$)} \label{defnzzcat}
 An object in $Z(Y,y)$ is a
sextuple $\Theta = (\mathcal{L},K,C,\alpha,\beta,\gamma)$ where $\mathcal{L} \in \mathbb{P}(Y^o)$, 
and $K$ and $C$ are vector spaces together with an exact sequence:
\begin{equation} \label{zzobj}
\begin{CD}
H^{n-1}(i^*j_* \mathcal{L}) @>\alpha>> K @>\beta>> C @>\gamma>> H^n(i^*j_* \mathcal{L})
\end{CD}
\end{equation}
  
Let $\Theta'=(\mathcal{L}',K',C'\alpha',\beta',\gamma') \in Obj(Z(Y,y))$.  
A morphism $\theta:\Theta \to \Theta'$ in $Z(Y,y)$ consists of
maps $\mathcal{L} \mapsto \mathcal{L}'$, $K \mapsto K'$ and $C \mapsto C'$
such that the following diagram commutes,
\begin{equation} \label{zzmor}
\begin{CD}
H^{n-1}(i^*j_* \mathcal{L}) @>\alpha>> K @>\beta>> C @>\gamma>> H^n(i^*j_* \mathcal{L}) \\ 
@VVV				@VVV	@VVV	@VVV\\
H^{n-1}(i^*j_* \mathcal{L}') @>\alpha'>> K'@>\beta'>>  C' @>\gamma'>> H^n(i^*j_* \mathcal{L}')
\end{CD}
\end{equation}
\end{defn}
\noindent
Remark: 
We will be primarily interested in the full sub-category $Z_\mathbb{Q}(Y,y)$ of $Z(Y,y)$ that 
consists of zig-zag objects $\Theta$ where the local system $\mathcal{L}$ is the constant sheaf $\uuline{\mathbb{Q}}$.
\begin{defn} \label{zigzagfunct}
The zig-zag functor $\mu:\mathbb{P}(Y) \to Z(Y,y)$ is defined by sending an object 
$\uuline{Q^\cdot} \in \mathbb{P}(Y)$ to the triple 
($j^* \uuline{Q^\cdot}$, $H^n(i^!\uuline{Q^\cdot})$, $H^n(i^*\uuline{Q^\cdot})$) 
together with the exact sequence
\begin{equation} \label{zzfunc}
\begin{CD}
H^{n-1}(i^*j_* j^*\uuline{Q^\cdot}) @>>> H^n(i^!\uuline{Q^\cdot}) @>>> H^n(i^*\uuline{Q^\cdot}) @>>> H^n(i^*j_* j^*\uuline{Q^\cdot}) 
\end{CD}
\end{equation}
%that comes from the distinguished triangle obtained by applying $j^*$ to a distinguished triangle.
\end{defn}
\end{subsubsection}
\begin{subsubsection}{Relationship to $\mathbb{P}(Y)$}
We will use Theorem 2.1 from MacPherson and Vilonen~\cite{mv:elem} in the proof of the main result.  It 
is stated below in a modified form.
\begin{thm} (MacPherson-Vilonen~\cite{mv:elem}) \label{MV}
\begin{enumerate}
\item 
The zig-zag functor $\mu:\mathbb{P}(Y) \to Z(Y,y)$ gives rise to a bijection from isomorphism classes of objects of $\mathbb{P}(Y)$ to isomorphism classes of 
objects of $Z(Y,y)$,
\item
Given $\uuline{\mathcal{S}},\uuline{\mathcal{S}'} \in \mathbb{P}(Y)$.  
Then $\mu:Hom_{\mathbb{P}}(\uuline{\mathcal{S}},\uuline{\mathcal{S}'}) \to Hom_{Z}(\mu(\uuline{\mathcal{S}}),\mu(\uuline{\mathcal{S}'}))$ is a surjection.
\end{enumerate}
\end{thm}
\noindent
Hence an isomorphism class of objects in $\mathbb{P}(Y)$ is determined by a local system (in this case the constant local system) on $Y^o$ and a fixed $K$ and $C$.  Different choices of $K$ and $C$ lead to different perverse sheaves on $Y$.  The main result of this work involves constructing a certain perverse sheaf on $Y$ from a certain choice of $K$ and $C$, as described in the following proposition.
\begin{prop} \label{defnKC}
Let $\Theta_0=(\uuline{\mathbb{Q}},K_0,C_0,\alpha_0,\beta_0,\gamma_0)$ 
where $K_0=Im(H^n_c(c^o L) \to H^n_c(Y^o))$, $C_0=Im(H^n(Y^0) \to H^{n+1}_c(c^o L))$.  Let 
$\alpha_0:H^n_c(c^o L) \to Im(H^n_c(c^o L) \to H^n_c(Y^o))$.  Let $\beta_0$ be the 0-map.  Let
$\gamma_0:Im(H^n(Y^0) \to H^{n+1}_c(c^o L)) \to H^{n+1}_c(c^o L)$.
Then $\Theta_0 \in Obj(Z_\mathbb{Q}(Y,y))$.  Therefore there exists $\uuline{\mathcal{S}_0} \in \mathbb{P}_\mathbb{Q}(Y)$ such that $\mu(\uuline{\mathcal{S}_0}) = \Theta_0$. 
\end{prop}  
\begin{proof}
In order to have an object of $Z_\mathbb{Q}(Y,y)$, we must specify $K_0$, $C_0$, the maps between them and 
then show eqn. (\ref{zzobj}) is exact. Rewriting eqn. (\ref{zzobj}) yields,
\begin{equation} \label{seqkczero}
\begin{CD}
H^n_c(c^oL) @>\alpha_0>> K_0 @>\beta_0>> C_0 @>\gamma_0>> H^{n+1}_c(c^oL)
\end{CD}
\end{equation}
Recall the distinguished triangle in eqn. (\ref{d1}) applied to $Y$ with coefficients in $\uuline{\mathbb{Q}}$.
Simplifying in degree n we have,
\begin{equation} \label{cohY}
\begin{CD}
@>>> H^n_c(c^oL) @>>> H^n_c(Y^o) @>>> H^n(Y^o) @>>> H^{n+1}_c(c^oL) @>>>
\end{CD}
\end{equation}
Comparing eqns. (\ref{seqkczero}) and (\ref{cohY}) can be identified as 
$\alpha_0: H^n_c(c^oL;\uuline{\mathbb{Q}}) \to K_0$ a surjection, $\gamma_0:C_0 \to H^{n+1}_c(c^oL;\uuline{\mathbb{Q}})$ 
an injection which are canonical maps that follow by definition.  
Since eqn. (\ref{cohY}) is exact with $K_0 \subset H^n_c(Y;\uuline{\mathbb{Q}})$ and
$C_0 \subset H^{n+1}_c(c^oL;\uuline{\mathbb{Q}})$, $\beta_0:K_0 \to C_0$ is the zero map since moving across two elements of an exact sequence is equivalent to applying $d \circ d =0$.
All that remains is to show that eqn. (\ref{seqkczero}) is exact. Now, since $\beta_0$ is the 
zero map it follows that 
$Ker(\beta_0)=Im(\beta_0)=0$. Since $\alpha_0$ is a surjection, $Ker(\beta_0)=Im(\alpha_0)=0$.  
Similarly since $\gamma_0$ is an injection,
$Ker(\gamma_0)=Im(\beta_0)=0$, hence (\ref{seqkczero}) is exact and $\Theta_0 \in Obj(Z_\mathbb{Q}(Y,y))$.
Since, $\Theta_0 \in Obj(Z_\mathbb{Q}(Y,y))$, it follows by Theorem \ref{MV}
there exists $\uuline{\mathcal{S}_0} \in \mathbb{P}_\mathbb{Q}(Y)$ such that $\mu(\uuline{\mathcal{S}_0}) = \Theta_0$.
\end{proof}
\end{subsubsection}
\end{subsection}
\end{section}
\section{Main Result}
In this section we state the main result and discuss its proof.
\begin{thm}\label{mainthm}
The perverse sheaf $\uuline{\mathcal{S}_0}$ has the following properties:
\begin{enumerate}
\item
$H^i(Y;\uuline{\mathcal{S}_0})=
\begin{cases}
H^i(Y),		 &\text{i $>$ n},\\
H^i(Y^o),	 &\text{i $<$ n}
\end{cases}
$
\item
$H^n(Y;\uuline{\mathcal{S}_0})$ is specified by the following two canonical short exact sequences:
\begin{enumerate}
\item
$0 \to K_0 \to  H^n(Y;\uuline{\mathcal{S}_0}) \to H^n(Y^o) \to 0$
\item
$0 \to H^n_c(Y^o) \to H^n(Y;\uuline{\mathcal{S}_0}) \to C_0 \to 0$
\end{enumerate}
\item
$\uuline{\mathcal{S}_0}$ is self-dual. 
\end{enumerate}
\end{thm}
\noindent
Remark: The two short exact sequences in (2) are equivalent to the diagram 
\begin{equation} \label{origtri}
\xymatrix{
& H^n(Y;\uuline{\mathcal{S}_0}) \ar@{->>}[dr]^d \\
H^n(Y;\uuline{\mathbb{Q}}) \ar@{^{(}->}[ur]^c  \ar[rr]^\gamma  & & H^{n}(Y;j_*j^*\uuline{\mathbb{Q}})}
\end{equation}
\noindent
where $c$ is an injection, $d$ is a surjection, and $\gamma$ is the map that comes from 
the distinguished triangle between these elements.
\begin{proof}
(Parts 1 and 2)\\
Recall that $H^n(Y;\uuline{\mathcal{S}_0})$ fits into the long exact sequence generated by applying the distinguished 
triangle (\ref{d2}) to $Y$ on $\uuline{\mathbb{Q}}$.  We will prove there is a diagram
\begin{equation} \label{d3}
\xymatrix{
& & & &  H^n(Y;\mathcal{S}_0) \ar@{->>}[d]^d \\ 
  & H^{n-1}(j_*j^*\mathbb{Q})\ar[r]^\alpha & H^{n}(i_*i^!\mathbb{Q})\ar[r]^\beta \ar@{->>}[d]^a  
	&  H^n(\mathbb{Q}) \ar@{^{(}->}[ur]^c \ar[r]^\gamma
    & H^{n}(j_*j^*\mathbb{Q}) \ar[r]^\delta \ar@{->>}[d]^e & H^{n+1}(i_*i^!\mathbb{Q}) \\
	& & K_0=Im(\beta) \ar@{^{(}->}[ur]^b  & & C_0=Im(\delta) \ar@{^{(}->}[ur]^f }
\end{equation}
We will construct the maps for this triangle and show they exist. 
Recall the support and cosupport conditions for a perverse sheaf described in Definition \ref{defnpervshf}.  
Consider the distinguished triangle (\ref{d1}) applied to $Y$ with values in the perverse sheaf $\uuline{\mathcal{S}_0}$.
Explicitly, this is $H^k(Y;i_*i^!\uuline{\mathcal{S}_0}) \to H^k(Y;\uuline{\mathcal{S}_0}) \to H^k(Y;j_*j^*\uuline{\mathcal{S}_0})$ 
(see eqn. (\ref{d1coh})).
Recall that, $H^k(Y;j_*j^*\uuline{\mathcal{S}_0}) \cong H^k(Y^o;\mathbb{Q})\simeq H^k(Y^o)$. 
Then applying the cosupport condition in Definition \ref{defnpervshf}, $H^k(Y;i_*i^!\uuline{\mathcal{S}_0})=0$ which implies that
\begin{equation} \label{HLowDeg1}
H^i(Y;\uuline{\mathcal{S}_0}) \simeq H^i(Y^o) \text{ } (0 < i < n-1).
\end{equation}
Consider once again the distinguished triangle (\ref{d1}) applied to $Y$ with values in the 
perverse sheaf $\uuline{\mathcal{S}_0}$ in the following degrees,
\begin{equation} \label{ExSeqDeg1}
\begin{CD}
H^{n-1}(Y;i_*i^!\uuline{\mathcal{S}_0}) @>>> H^{n-1}(Y;\uuline{\mathcal{S}_0}) @>>> H^{n-1}(Y;j_*j^*\uuline{\mathcal{S}_0}) @>>> H^n(Y;i_*i^!\uuline{\mathcal{S}_0})
\end{CD}
\end{equation}
By the cosupport condition in Definition \ref{defnpervshf}, $H^{n-1}(Y;i_*i^!\uuline{\mathcal{S}_0})=0$.  So (\ref{ExSeqDeg1}) becomes,
\begin{equation} \label{ExSeqDeg2}
\begin{CD}
0 @>>> H^{n-1}(Y;\uuline{\mathcal{S}_0}) @>>> H^{n-1}(Y;j_*j^*\uuline{\mathcal{S}_0}) @>>> H^n(Y;i_*i^!\uuline{\mathcal{S}_0})
\end{CD}
\end{equation}
Identifying the maps in (\ref{ExSeqDeg2}) we have,
\begin{equation} \label{ExSeqDeg3}
\begin{CD}
0 @>>> H^{n-1}(Y;\uuline{\mathcal{S}_0}) @>\theta>> H^{n-1}(Y;j_*j^*\uuline{\mathcal{S}_0}) @>\phi>> H^n(Y;i_*i^!\uuline{\mathcal{S}_0})
\end{CD}
\end{equation}
By inspection of (\ref{ExSeqDeg3}), $\theta$ is an injection.  Recall that $H^n(Y;i_*i^!\uuline{\mathcal{S}_0})=H^n_c(c^o L;\uuline{\mathcal{S}_0})=K_0$.
So (\ref{ExSeqDeg3}) can be rewritten as 
\begin{equation} \label{ExSeqDeg4}
\begin{CD}
0 @>>> H^{n-1}(Y;\uuline{\mathcal{S}_0}) @>\theta>> H^{n-1}(Y;j_*j^*\uuline{\mathcal{S}_0}) @>\phi>> K_0
\end{CD}
\end{equation}
Recall Proposition \ref{defnKC}, where we defined $K_0=Im(\beta)\subset H^n(\mathbb{Q})$. This implies that
$\phi=\alpha \circ \beta = 0$ since the horizontal part of (\ref{d3}) is part 
of a long exact sequence.  Hence $\phi$ is the
zero map which means that $\theta$ is an isomorphism. The result is that
$H^{n-1}(Y;\uuline{\mathcal{S}_0}) \simeq H^{n-1}(Y;j_*j^*\uuline{\mathcal{S}_0})$.
This means that (\ref{HLowDeg1}) can be rewritten to include
degree $n-1$ as,
\begin{equation} \label{HLowDeg2}
H^i(Y;\mathcal{S}_0) \simeq H^i(Y^o) \text{ } (0 < i \leq n-1).
\end{equation}
\\
Consider the distinguished triangle (\ref{d2}) applied to $Y$ with values in the perverse sheave $\uuline{\mathcal{S}_0}$.
Then for $i>n+1$, by the support condition in Definition \ref{defnpervshf}, $H^k(Y;i_*i^*\uuline{\mathcal{S}_0})=0$ which implies that
$H^k(Y;\uuline{\mathcal{S}_0}) \simeq H^k(Y;j_!j^*\uuline{\mathcal{S}_0}) \simeq H^k_c(Y^o) \simeq H^k(Y,c^o L;\uuline{\mathcal{S}_0}) \simeq H^k(Y)$.
It follows that,
\begin{equation}\label{HHiDeg1}
H^i(Y;\uuline{\mathcal{S}_0}) \simeq H^i(Y) \text{ } (i > n+1)
\end{equation}
Consider once again the distinguished triangle (\ref{d2}) applied to $Y$ with values in the 
perverse sheaf $\uuline{\mathcal{S}_0}$ in the following degrees,
\begin{equation} \label{ExSeqDegs3}
\begin{CD}
H^{n}(Y;i_*i^*\uuline{\mathcal{S}_0}) @>>> H^{n+1}(Y;j_!j^*\uuline{\mathcal{S}_0}) @>>> H^{n+1}(Y;\uuline{\mathcal{S}_0}) @>>> H^{n+1}(Y;i_*i^*\uuline{\mathcal{S}_0})
\end{CD}
\end{equation}
By the support condition in Definition \ref{defnpervshf}, $H^{n+1}(Y;i_*i^*\uuline{\mathcal{S}_0})=0$.  So (\ref{ExSeqDegs3}) becomes,
\begin{equation} \label{ExSeqDegs4}
\begin{CD}
H^{n}(Y;i_*i^*\uuline{\mathcal{S}_0}) @>>> H^{n+1}(Y;j_!j^*\uuline{\mathcal{S}_0}) @>>> H^{n+1}(Y;\uuline{\mathcal{S}_0}) @>>> 0
\end{CD}
\end{equation}
Identifying the maps in (\ref{ExSeqDegs4}) we have,
\begin{equation} \label{ExSeqDeg5}
\begin{CD}
H^{n}(Y;i_*i^*\uuline{\mathcal{S}_0}) @>\omega>> H^{n+1}(Y;j_!j^*\uuline{\mathcal{S}_0}) @>\tau>> H^{n+1}(Y;\uuline{\mathcal{S}_0}) @>>>  0
\end{CD}
\end{equation}
By inspection of (\ref{ExSeqDeg5}) $\tau$ is a surjection.  Recall that $H^n(Y;i_*i^*\uuline{\mathcal{S}_0}) = H^n(c^o L;\uuline{\mathcal{S}_0}) = C_0$.
\begin{equation} \label{ExSeqDeg6}
\begin{CD}
C_0 @>\omega>> H^{n+1}(Y;j_!j^*\uuline{\mathcal{S}_0}) @>\tau>> H^{n+1}(Y;\uuline{\mathcal{S}_0}) @>>> 0
\end{CD}
\end{equation}
Recall Proposition \ref{defnKC}, where we defined $C_0=im(\delta)$.  Define $\mu:C_0 \to H^{n+1}(Y;\uuline{\mathcal{S}_0})$. By exactness it follows that $C_0=ker(\mu)=im(\delta)$ which is a subgroup of $H^{n+1}(i_*i^!\mathbb{Q})$.
Since every element in $C_0$ maps to 0 in $H^{n+1}(Y;j_!j^*\uuline{\mathcal{S}_0}) \cong H^{n+1}_c(Y^o) \cong H^{n+1}(Y;\uuline{\mathcal{S}_0})$ the map from $\mu$
is the $0$-map and thus $\tau$ is a bijection.  It follows that, we can 
rewrite (\ref{HHiDeg1}) as
\begin{equation}\label{HHiDeg2}
H^i(Y;\uuline{\mathcal{S}_0}) \simeq H^i(Y) \text{ } (i \geq  n+1)
\end{equation}
Identifying the maps in the diagram (\ref{d3}), we have
\begin{equation} \label{ShExSeq1}
\begin{CD}
0 @>>> K_0 @>{c \cdot b}>> H^n(\uuline{\mathcal{S}_0}) @>d>> H^n(j_*j^*\mathbb{Q}) @>>> & 0
\end{CD}
\end{equation}
We must show exactness at each term.  Since $c \cdot b$ is injective we have exactness at $K_0$.  Since
$d$ is surjective we have exactness at $H^n(j_*j^*\uuline{\mathbb{Q}})$.  All that remains is to show exactness at
$H^n(\uuline{\mathcal{S}_0})$.  Recall that $K_0=ker(\gamma)$.  Take $x \in K_0$, then 
$d((c \cdot b)(x))= \gamma(x) = 0$.  This shows that $im(c \cdot b) \subset ker(d)$.  It remains to show that
$ker(d) \subset im(c \cdot b)$.  Take $x \in ker(d)$.  We want to show $\exists$ $y \in H^n(\mathbb{Q})$ such that
$(c \cdot b)(y)=x$.  Let $y \in H^n(\mathbb{Q})$, $x \in ker(d)$, and take $\gamma(y)=d(x)$.  Recall that we required the map $d$ to be surjective, 
$(c \cdot b)(y)=x$ which means that $ker(d) \subset im(c \cdot b)$ hence, we have exactness at $H^n(\mathcal{S}_0)$.

Identifying the maps in the diagram (\ref{d3}), we have
\begin{equation} \label{ShExSeq2}
\begin{CD}
0 @>>> H^n(\mathbb{Q}) @>c>> H^n(\mathcal{S}_0) @>{e \cdot d}>> C_0 @>>> 0
\end{CD}
\end{equation}
We must show exactness at each term.  Since $c$ is an injection we have exactness at $H^n(\mathbb{Q})$ and
since $e \cdot d$ is surjective we have exactness at $C_0$.  All that remains is to show exactness at $H^n(\uuline{\mathcal{S}_0})$.
Recall that $C_0= im(\delta)=ker(\mu)$ which is a subgroup of $H^{n+1}(i_*i^!\mathbb{Q})$.  Take $x \in H^n(\mathbb{Q})$ then $(e \cdot d)(c(x))= \delta(\gamma(x)) =0$.  This shows that 
$im(c) \subset ker(e \cdot d)$.  It remains to show that $ker(e \cdot d)  \subset im(c)$.  Take $y \in ker(e \cdot d)$.
We want to show $\exists$ $z \in H^n(\mathbb{Q})$ such that $c(z)=y$.  Let $z \in H^n(\mathbb{Q})$, $y \in ker(e \cdot d)$, 
and take $\delta \cdot \gamma(z)=e \cdot d (y)$.  The map $e \cdot d$ is surjective since both e and d are surjective.  Then 
$y=c(z)$ which means that $ker(e \cdot d)  \subset im(c)$ hence we have exactness at $H^n(\uuline{\mathcal{S}_0})$.
This completes the proof of parts 1 and 2.
\end{proof}
In the next section we will present the proof of part 3 of Theorem \ref{mainthm}.
\section{Duality}
\subsection{Overview and Definitions}
The goal of this section is to prove part three of the main result.  
We will use part two of Theorem \ref{MV} (MacPherson-Vilonen) 
to prove that $\uuline{\mathcal{S}_0}$ is self-dual in $P_\mathbb{Q}(Y)$.  
The organization of this section will be around two main themes.
The first goal will be to construct a duality functor in $Z(Y,y)$ and then show that the object 
defined in Definition \ref{defnKC} is self-dual.  Then Theorem \ref{MV} will generate
the corresponding self-dual object in $\mathbb{P}(Y)$.  Recall that we will use definitions
and statements from section \ref{conv}.

\subsection{Duality in $Z(Y,y)$}
We seek a duality functor $\mathcal{D}_Z$ in $Z(Y,y)$, 
that is compatible with the Zig-Zag functor $\mu:\mathbb{P}(Y) \to Z(Y,y)$ stated in
Definition \ref{zigzagfunct}.  We will need the following lemma.
\begin{lem} \label{dualident}
Given $\Theta=(\mathcal{L},K,C,\alpha,\beta,\gamma) \in Obj(Z(Y,y))$.  Then the 
morphism $\hat{\phi}$ from Proposition \ref{indmap} gives rise to isomorphisms:
\begin{enumerate}
\item
$H^{n-1}(i^*j_*\mathcal{L}^*[-2n]) \cong (H^n(i^*j_*\mathcal{L}))^*$
\item
$H^n(i^*j_*\mathcal{L}^*[-2n]) \cong (H^{n-1}(i^*j_*\mathcal{L}))^*$
\end{enumerate}
where $\mathcal{D}_V(\mathcal{L}) = \mathcal{L}^*$ and $()^*$ is the vector space dual.
\end{lem}
\noindent
Remark:  The following proof utilizes calculations from p. 185-186, Lemma 2.20 in sect. 20 from 
Ref.~\cite{gormac:wghtdcoh}.
\begin{proof}
The morphism $\hat{\phi}$ of Proposition \ref{indmap} gives rise to the following 
\begin{align}
H^{n-1}(i^*j_*\mathcal{L}^*[-2n])   &\cong H^{n-1}(i^*\mathcal{D}_V(j_!\mathcal{L}))[-2n]\notag\\
									&\cong H^{n-1}(\mathcal{D}_V(i^!j_!\mathcal{L}))[-2n]\notag\\
									&\cong H^{n+1}(\mathcal{D}_V(i^!j_!\mathcal{L}))\notag\\
									&\cong (H^n(i^*j_*\mathcal{L}))^*  
\end{align}
where $H^{n+1}(\mathcal{D}_V(i^!j_!\mathcal{L})) \cong (H^n(i^*j_*\mathcal{L}))^*$ using Lemma \ref{jdualident}.
\begin{align}
H^n(i^*j_*\mathcal{L}^*[-2n])       &\cong H^n(i^*\mathcal{D}_V(j_!\mathcal{L})[-2n])\notag\\
									&\cong H^n(\mathcal{D}_V(i^!j_!\mathcal{L}))[-2n]\notag\\
									&\cong H^{n}(\mathcal{D}_V(i^!j_!\mathcal{L}))\notag\\
									&\cong (H^{n-1}(i^*j_*\mathcal{L}))^*  
\end{align}
where $H^{n}(\mathcal{D}_V(i^!j_!\mathcal{L})) \cong (H^{n-1}(i^*j_*\mathcal{L}))^*$ using Lemma \ref{jdualident}.
\end{proof}

\noindent
We now define the dual of an object in $Z(Y,y)$.
\begin{defn} \label{dualzzobj}
Given $\Theta=(\mathcal{L},K,C,\alpha,\beta,\gamma) \in Obj(Z(Y,y))$.  
Define $\mathcal{D}_Z(\Theta)=(\mathcal{L}^*,C^*,K^*,\gamma^*,\beta^*,\alpha^*)$
where $\gamma^*:H^{n-1}(i^*j_* \mathcal{L}^*) \to C^*$ is the map dual to $\alpha$, $\beta^*:C^* \to K^*$
is the map dual to $\beta$, and $\alpha^*:K^* \to H^n(i^*j_* \mathcal{L}^*)$ is the map dual to $\gamma$.
\end{defn}
\noindent
Consider the object $\Theta=(\mathcal{L},K,C,\alpha,\beta,\gamma)$ with its exact sequence and maps identified as follows.
\begin{equation} \label{zzobjwmaps}
\begin{CD}
H^{n-1}(i^*j_* \mathcal{L}) @>{\alpha}>> K @>{\beta}>> C @>{\gamma}>> H^n(i^*j_* \mathcal{L})
\end{CD}
\end{equation}
  
\begin{lem} \label{dualobjzzcat}
Let $\Theta=(\mathcal{L},K,C,\alpha,\beta,\gamma) \in Obj(Z(Y,y))$ and 
$\mathcal{D}_Z(\Theta)=(\mathcal{L}^*,C^*,K^*,$\\$\gamma^*,\beta^*,\alpha^*)$ 
with the following maps,
\begin{equation} \label{zzobjstar}
\begin{CD}
H^{n-1}(i^*j_* \mathcal{L}^*) @>{\gamma^*}>> C^* @>{\beta^*}>> K^* @>{\alpha^*}>> H^n(i^*j_* \mathcal{L}^*)
\end{CD}
\end{equation}
where $\alpha^*$, $\beta^*$, and $\gamma^*$ are the dual maps defined in Definition \ref{dualzzobj}.  
Then the sequence (\ref{zzobjstar}) is exact and it follows that $\mathcal{D}_Z(\Theta)$ is an object of $Z(Y,y)$.
\end{lem}
\begin{proof}
The dual of the exact sequence in (\ref{zzobjwmaps}) is given by the exact sequence,
\begin{equation} \label{zzobjstar2}
\begin{CD}
(H^n(i^*j_* \mathcal{L}))^* @>{\gamma^*}>> C^* @>{\beta^*}>> K^* @>{\alpha^*}>> (H^{n-1}(i^*j_* \mathcal{L}))^*
\end{CD}
\end{equation}
By Lemma \ref{dualident}, since $(H^n(i^*j_* \mathcal{L}))^* \cong H^{n-1}(i^*j_* \mathcal{L}^*)$ and
$(H^{n-1}(i^*j_* \mathcal{L}))^* \cong H^n(i^*j_* \mathcal{L}^*)$, we can substitute these into (\ref{zzobjstar2})
which becomes exactly (\ref{zzobjstar}).  Since the sequence is exact it follows that $\mathcal{D}_Z(\Theta) \in Obj(Z(Y,y))$.
\end{proof}
\noindent
The duality functor $\mathcal{D}_Z$ in $Z(Y,y)$ needs to be compatible the Verdier dual in  $\mathbb{P}(Y)$.  The
following proposition demonstrates this.
\begin{prop} \label{diagPZ}
Let $\mu$ be the map as defined in Theorem \ref{MV}.  Then the following diagram commutes up to canonical
isomorphism:
\begin{equation} \label{commdiag}
\begin{CD}
\mathbb{P}(Y)			@>\mu>>	Z(Y,y)\\ 
@V{\mathcal{D}_V}VV		@V{\mathcal{D}_Z}VV\\
\mathbb{P}(Y)			@>\mu>>	Z(Y,y) 
\end{CD}
\end{equation}
\end{prop}
\begin{proof}
Let $\uuline{\mathcal{S}^\cdot} \in Obj(\mathbb{P}(Y))$.  We want to show that
$\mu(\mathcal{D}_V(\uuline{\mathcal{S}^\cdot})) \cong \mathcal{D}_Z(\mu(\uuline{\mathcal{S}^\cdot}))$.
Applying $\mu$ to 
$\uuline{\mathcal{S}^\cdot}$ gives an object in $Z(Y,y)$ which by Definition \ref{zzfunc} is a triple,
($j^* \uuline{\mathcal{S}^\cdot}$, $H^n(i^!\uuline{\mathcal{S}^\cdot})$, $H^n(i^*\uuline{\mathcal{S}^\cdot})$).
This triple has an associated exact sequence obtained by applying $i^*$ to the distinguished triangle (\ref{d1}),
\begin{equation} \label{zzfunc1}
\begin{CD}
H^{n-1}(i^*j_* j^*\uuline{\mathcal{S}^\cdot}) @>>> H^n(i^!\uuline{\mathcal{S}^\cdot}) @>>> H^n(i^*\uuline{\mathcal{S}^\cdot}) @>>> H^n(i^*j_* j^*\uuline{\mathcal{S}^\cdot}) 
\end{CD}
\end{equation}
Applying $\mathcal{D}_Z$ to the Zig-zag object $\mu(\uuline{\mathcal{S}^\cdot})$ makes the maps in eqn. (\ref{zzfunc1}) 
reverse direction with each term dualized.  The resulting Zig-zag object, $\mathcal{D}_Z(\mu(\uuline{\mathcal{S}^\cdot}))$, can be expressed as 
($(j^* \uuline{\mathcal{S}^\cdot})^*$, $(H^n(i^*\uuline{\mathcal{S}^\cdot}))^*$,$(H^n(i^!\uuline{\mathcal{S}^\cdot}))^*$) which has the following exact sequence,
\begin{equation} \label{zzfunc2}
\begin{CD}
(H^n(i^*j_* j^*\uuline{\mathcal{S}^\cdot}))^* @>>> (H^n(i^*\uuline{\mathcal{S}^\cdot}))^* @>>> (H^n(i^!\uuline{\mathcal{S}^\cdot}))^* @>>> (H^{n-1}(i^*j_* j^*\uuline{\mathcal{S}^\cdot}))^*
\end{CD}
\end{equation}
Now consider the Zig-zag object $\mu(\mathcal{D}_V(\uuline{\mathcal{S}^\cdot}))$.  This is a 
triple ($j^* \mathcal{D}_V(\uuline{\mathcal{S}^\cdot})$, $H^n(i^!\mathcal{D}_V(\uuline{\mathcal{S}^\cdot})$, $H^n(i^*\mathcal{D}_V(\uuline{\mathcal{S}^\cdot})$) with exact sequence
\begin{equation} \label{zzfunc3}
\begin{CD}
H^{n-1}(i^*j_* j^*\mathcal{D}_V(\uuline{\mathcal{S}^\cdot})) @>>> H^n(i^!\mathcal{D}_V(\uuline{\mathcal{S}^\cdot})) @>>> H^n(i^*\mathcal{D}_V(\uuline{\mathcal{S}^\cdot})) @>>> H^n(i^*j_* j^*\mathcal{D}_V(\uuline{\mathcal{S}^\cdot}))
\end{CD}
\end{equation}
Simplifying the object and the exact sequence in eqn. (\ref{zzfunc3}) using Lemma \ref{dualident} we have the triple
($(j^* \uuline{\mathcal{S}^\cdot})^*$, $(H^n(i^*\uuline{\mathcal{S}^\cdot}))^*$,$(H^n(i^!\uuline{\mathcal{S}^\cdot}))^*$) with exact sequence,
\begin{equation} \label{zzfunc4}
\begin{CD}
(H^n(i^*j_* j^*\uuline{\mathcal{S}^\cdot}))^* @>>> (H^n(i^*\uuline{\mathcal{S}^\cdot}))^* @>>> (H^n(i^!\uuline{\mathcal{S}^\cdot}))^* @>>> (H^{n-1}(i^*j_* j^*\uuline{\mathcal{S}^\cdot}))^*
\end{CD}
\end{equation}
which is exactly the same as the Zig-zag object $\mathcal{D}_Z(\mu(\uuline{\mathcal{S}^\cdot}))$.  Showing the diagram commutes on morphisms is a similar argument and will be left to the reader.
\end{proof}
\noindent
We will need the following lemma from linear algebra.
\begin{lem}\label{dualimg}{(\textbf{Duals of Images})}
Given A, B, C, and D are vector spaces with maps $f:A \to B$ and $g:C \to D$.  Let $<,>_1:A \times D \to \mathbb{Q}$ 
and $<,>_2:B \times C \to \mathbb{Q}$ be non-degenerate pairings such that 
$<a,g(c)>_1=<f(a),c>_2$ $\forall a \in A \text{ and } \forall c \in C$.  Then $<,>_1$ and $<,>_2$ induce a non-degenerate pairing $<,>_3:Im(f) \times Im(g) \to \mathbb{Q}$.
\end{lem}
\begin{thm} \label{mainresult}
The object $\Theta_0$ as defined in Proposition \ref{defnKC} is self-dual in $Z_\mathbb{Q}(Y,y)$.
\end{thm}
\begin{proof}
We need to construct an isomorphism between $\Theta_0$ and its dual $\mathcal{D}_Z(\Theta_0)$ in $Z_\mathbb{Q}(Y,y)$ 
where $\Theta_0=(\uuline{\mathbb{Q}},K_0,C_0,\alpha_0,\beta_0,\gamma_0)$
and $\mathcal{D}_Z(\Theta_0)=(\uuline{\mathbb{Q}}^*,C^*_0,K^*_0,\alpha_0^*,\beta_0^*,\gamma_0^*)$.  In order to construct an isomorphism
in $Z_\mathbb{Q}(Y,y)$, recall the definition of a morphism in $Z_\mathbb{Q}(Y,y)$ (eqn. (\ref{zzmor})) as a map from $\Theta_0 \to \mathcal{D}_Z(\Theta_0)$ where the exact sequence for $\Theta_0$ maps isomorphically to the exact
sequence for $\mathcal{D}_Z(\Theta_0)$ and $\uuline{\mathbb{Q}}$ maps isomorphically to $\uuline{\mathbb{Q}}^*$.  These exact sequence 
maps are labelled in the following diagram.  
\begin{equation} \label{sddiag}
\begin{CD}
H^{n-1}(i^*j_* \mathcal{L}_0) @>f>> K_0 @>g>> C_0 @>h>> H^n(i^*j_* \mathcal{L}_0)\\ 
	@V{\kappa}VV	@V{\lambda}VV	@V{\nu}VV	@V{\xi}VV\\
H^{n-1}(i^*j_* \mathcal{L}_0^*[-2n]) @>l>> C^*_0 @>m>> K^*_0 @>n>> H^n(i^*j_* \mathcal{L}_0^*[-2n])
\end{CD}
\end{equation}
where the maps $\kappa, \lambda, \nu, \text{and } \xi$ are non-degenerate pairings of isomorphism classes
 defined as follows: $\kappa:<,>_\kappa$, $\lambda:<,>_\lambda$, $\nu:<,>_\nu$, and $\xi:<,>_\xi$. 
Before we begin constructing the vertical isomorphisms defined in (\ref{sddiag}), 
we must show that there is a map from $\uuline{\mathbb{Q}} \to \uuline{\mathbb{Q}}^*$.  The map $\uuline{\mathbb{Q}} \otimes \uuline{\mathbb{Q}} \to
\uuline{\mathbb{Q}}$ can be described by multiplication which implies that $\uuline{\mathbb{Q}}$
is dual to itself.
The remainder of the proof will be concerned with showing that the vertical maps 
$\kappa, \lambda, \nu, \text{and } \xi$ between the 
exact sequences are isomorphisms and that the diagram (\ref{sddiag}) commutes.  
By Lemma \ref{dualident}, $\kappa$ and $\xi$ are isomorphisms.  
Next, we must show that $\lambda$ and $\nu$ are isomorphisms.  Recall that 
$K_0=Im(H^n_c(c^o L) \to H^n_c(Y^0))$ and $C_0=Im(H^n(Y^0) \to H^{n+1}_c(c^o L))$ are vector spaces.
Using Proposition \ref{indmap} we have duality isomorphisms $H^n_c(c^o L)$ to $H^{n+1}_c(c^o L)$ and $H^n(Y^0)$ to  
$H^n_c(Y^0)$.\footnote{By use of distinguished triangles, $H^n_c(c^o L) \simeq H^{n-1}(L)$ where $dim(L)=2n-1$. This implies that $H^{n-1}(L) \to H^n(L)$ are dual since $L$ is compact.  Thus,
applying Proposition \ref{indmap}, $H^n_c(c^o L) \to H^{n+1}_c(c^o L)$ is a duality isomorphism.  
In addition, $dim(Y^0)=2n$ so it follows in a straight forward manner 
that $H^n(Y^0)$ and $H^n_c(Y^0)$ are dual also.}  This implies there are two non-degenerate pairings 
$<,>_1:H^n_c(c^o L) \times H^{n+1}_c(c^o L) \to \mathbb{Q}$
and $<,>_2:H^n(Y^0) \times H^n_c(Y^0) \to \mathbb{Q}$.  By use of Lemma \ref{dualimg} on duals of 
images, $K_0 \times C_0 \to \mathbb{Q}$ is a non-degenerate pairing and $\lambda: K_0 \to C^*_0$ and
$\nu:C_0 \to K^*_0$ are isomorphisms.  Hence, $f$ is the dual map to $n$, $g$ is the dual map to $m$,
and $h$ is the dual map to $l$.
\\
\indent 
It remains to show that the diagram commutes.   In order to show the diagram commutes,
let $x \in H^{n-1}(i^*j_*\mathcal{L}_0)$ and $c \in C_0$.  We want to show that
$\lambda(f(x))(c) = l(\kappa(x))(c)$.  This is the same as showing 
$<f(x),c>_\lambda=<x,h(c)>_\kappa$ $\forall$ $x \in H^{n-1}(i^*j_*\mathcal{L}_0)$ and $c \in C_0$.  Since $\lambda$ is 
a map of isomorphism classes, $\forall$ $y \in C^*_0$ the lift of $<f(x),y>_\lambda$ is $f(x)$.  
Recall that $h=l^*$.  Consider $<x,h(c)>_\kappa \in C^*_0$.  Then since
$h:C_0 \to H^n(i^*j_*\mathcal{L}_0)$, for $a \in H^{n-1}(i^*j_*\mathcal{L}_0^*)$, $h(c)=a=l^*(a)$.  So
$l(<x,a>_\kappa)=<x,l^*(a)>_\kappa=<x,h(c)>_\kappa$ where $<x,a>_\kappa \in H^{n-1}(i^*j_*\mathcal{L}_0^*)$.
It follows that a lift of $<x,a>_\kappa$ under $\kappa$ is $x$.  This shows that $\lambda(f(x))(c) = l(\kappa(x))(c)$.

Next we must show that for $k',k \in K_0$, $m(\lambda(k))(k')=\nu(g(k))(k')$.  This is equivalent to
showing that $<k,g(k')>_\lambda=<g(k),k'>_\nu$.  The lift of 
$<g(k),k'>_\nu$ under $\nu$ is $g(k)$.  Under $g$, $k \mapsto g(k)$.  Recall that $g=m^*$. Let $c=\in C^*_0$.  Then since
$g:K_0 \to C_0$, $g(k)=c=m^*(c)$.  It follows that $m(<k,c>_\lambda)=<k,m^*(c)>_\lambda=<k,g(k)>_\lambda$ where
$<k,g(k)>_\lambda \in C^*_0$.  The lift of $<k,g(k)>_\lambda=k$.  This
shows that $m(\lambda(k))(k')=\nu(g(k))(k')$.

The final part of showing the diagram commutes is to verify that for $y \in H^n(i^*j_*\mathcal{L}_0^*)$
and $c \in C_0$, $<c,f(y)>_\nu=<h(c),y>_\xi$.  This is equivalent to showing that 
$n(\xi(c))(y)=\xi(h(c))(y)$.  Recall that $\xi$ is a map of isomorphism classes.
The lift of $<h(c),y>_\xi=h(c)$.  Under $h$,$c \mapsto h(c)$.  Recall that 
$n=f^*$.  Let $k \in K^*$, then 
$f(y)=k^*(k)$.  Consider $<c,k>_\nu \in K^*_0$.  It follows that 
$<c,f(y)>_\nu=<c,n^*(k)>_\nu(<c,k>_\nu)$.  Now a lift of $<c,k>_\nu =c$.
This shows that $<c,f(y)>_\nu=<h(c),y>_\xi$, and hence the diagram (\ref{sddiag})
commutes.
\end{proof}
\begin{cor} \label{Sselfdual}
The perverse sheaf $\uuline{\mathcal{S}_0}$ is self-dual in $\mathbb{P}_\mathbb{Q}(Y)$.
\end{cor}
\begin{proof}
By Theorem \ref{MV}, 
$\mu: Hom_\mathbb{P}(\uuline{\mathcal{S}_0},\mathcal{D}_V(\uuline{\mathcal{S}_0})) \to Hom_Z(\mu(\uuline{\mathcal{S}_0}),\mu(\mathcal{D}_V(\uuline{\mathcal{S}_0})))$ 
is a surjection. Since the diagram in (\ref{commdiag}) commutes,
$Hom_Z(\mu(\uuline{\mathcal{S}_0}),\mu(\mathcal{D}_V(\uuline{\mathcal{S}_0})))=$ $Hom_Z(\Theta_0,\mathcal{D}_Z(\Theta_0))$.
Now by Theorem \ref{mainresult}, we have constructed an isomorphism $\Phi \in Hom_Z(\Theta_0,\mathcal{D}_Z(\Theta_0))$.
Since $\mu$ is a surjection on morphisms, there exists an isomorphism $\bar{\Phi} \in Hom_\mathbb{P}(\uuline{\mathcal{S}_0},\mathcal{D}_V(\uuline{\mathcal{S}_0}))$.
Therefore $\uuline{\mathcal{S}_0}$ is self-dual in $\mathbb{P}_\mathbb{Q}(Y)$.
\end{proof}
\noindent
Remark: It is not known whether the isomorphism $\uuline{\mathcal{S}_0} \to \mathcal{D}_V(\uuline{\mathcal{S}_0})$
is unique.  It is conceivable that there may be several, essentially different, pairings
$\uuline{\mathcal{S}_0} \otimes \uuline{\mathcal{S}_0} \to \mathbb{D}_Y$.
\noindent
Corollary \ref{Sselfdual} completes the proof of the main result presented in Theorem \ref{mainthm}.
A direct corollary to Corollary \ref{Sselfdual} is Poincar\'e Duality.  We state this below.
\begin{cor}{(Poincar\'e Duality)} \label{poindual}
For all degrees $i \geq 0$, $H^i(Y;\uuline{\mathcal{S}_0}) \cong H^{2n-i}(Y;\uuline{\mathcal{S}_0})$.
\end{cor}
\section{An Example}
\subsection{The Construction}
We will look at a very simple construction of space time $\mathcal{M}  = X^{3,1} \times Y$ where $Y$ is a single 
node Calabi-Yau manifold (simple stratified space) as presented on pp. 276-277 in \cite{hubsch:best}.  
We will restate this construction here. 
%:TH: Beginning of changes
Consider the family of quintic hypersurfaces in $\mathbb{P}^4$, defined by
\begin{equation}
%:TH: added \label{e:1nodeY} and changed equation
 I:=x_5^3 (\Sigma_{i=1}^4 x_i^2) + \Sigma_{i=1}^5 a_i x_i^5 ~=~0, \label{e:1nodeY}
\end{equation}
where $x_i$ are the homogeneous coordinates of $\mathbb{P}^4$. 
%:TH: Corrected through the end of this section, added pointers to stringy work
On the $x_5 = 0$ hyperplane $\mathbb{P}^3 \subset \mathbb{P}^4$ we have that
\begin{equation}
I=\Sigma_{i=1}^4 a_i x_i^5,
\end{equation}
\begin{equation}
dI=5\Sigma_{i=1}^4 dx_i \cdot a_i x_i^4.
\end{equation}
For generic $a_i$, $I=0=dI$ only at $x_i=0$ which is not in $\mathbb{P}^4$, so the singular points of $Y$ are all in 
the $x_5 \neq 0$ coordinate patch and we set $x_5=1$.  For generic choice of $a_i$, $I=0$ and $dI=0$ have no common 
solution, so that the quintic hypersurface $I=0$ in $\mathbb{P}^4$ is smooth.

Let $a_5 \rightarrow 0$.  Then,
$dI=0$ implies $x_i(5a_ix_i^3 +2) =0$ and candidate singular points $(p^{\#})$ are parametrized as:
\begin{equation}
x_i = - \xi_i \sqrt[3]{\frac{2}{5a_i}} \cdot \omega^{k_i}
\end{equation}
where $\omega =e^{\frac{2i\pi}{3}}$ for $k_i=0,1,2$, and $\xi_i=0,1$ for $i=1,2,3,4$.  At these points
\begin{equation}
I(p^{\#}) = \frac{3}{5} \sqrt[3]{\frac{4}{25}} \Sigma_{i=1}^4 \frac{\xi_i\,\omega^{2k_i}}{a_i^{\frac{2}{3}}},
\end{equation}
the vanishing of which brings about several cases. The case of interest for the purposes of this paper is when $I(p^\#)$ vanishes at a single point, $(0,0,0,0,1)\in\IP^4$, i.e., setting $\xi_i=0$ $i=1,2,3,4$, whereupon both $I=0$ and $dI=0$---regardless of the choice of $a_i$'s.  There, $\det I''=16 \neq 0$ and this singular point is a node. 

We now fix some generic choice of $a_1,a_2,a_3,a_4$ and regard eqn.~(\ref{e:1nodeY}) as a pencil of quintics in
 $\IP^4$, parametrized by $a_5$. The quintic hypersurfaces in $\IP^4$ defined by eqn.~(\ref{e:1nodeY}) for each
 $a_5\neq0$ and $|a_5|$ not too large\footnote{Given concrete values of $a_1,a_2,a_3,a_4$, there may well exist an
 upper bound on $|a_5|$ for this to be true.} are all smooth; the one at $a_5$ is however singular and has a single
 node. Let $Y_{a_5}$ denote these smooth quintics, for $a_5\neq0$; let $Y=Y_0$ denote the 1-node singular quintic, at
$a_5=0$.

\subsection{Computing $H^*(Y)$}
We now focus on the quintic with a single node, $Y=Y_0$ where $Y^o$ denotes the non-singular part of $Y$. Let
 $\widetilde{Y}$ be the small resolution of $Y$; there is no obstruction to local surgery, replacing the node by a
 $\IP^1$. Note that $\widetilde{Y}$ cannot be K\"ahler: any putative K\"ahler form would have to be null 
on the exceptional $\IP^1$~\cite{hubsch:best}.\footnote{The single node example presented here has a major drawback for
application in superstring theory: neither $Y=Y_0$ nor its small resolution $\widetilde Y$ admit a K\"ahler metric. 
 However, the purpose of this example is to demonstrate that the complex of sheaves $\uuline{\mathbb{S}_0}$ 
can provide the needed ranks in all dimensions (as suggested by Hubsch \cite{hub:singspace} and Strominger 
\cite{stro:massless}).}\\
\indent We want to compute $H^*$ using $\uuline{\mathbb{S}_0}$ and compare it to the results obtained from string
theory. For $n \neq 3$, $H^n(Y;\uuline{\mathbb{S}_0}) \simeq H^n(Y;\uuline{IC}) \simeq
H^n(\widetilde{Y};\uuline{\mathbb{Q}})$ where $\widetilde{Y}$ is the small resolution.  In the middle
dimension we must use part 2 (short exact sequences) of Theorem 3.1.  Recall that
$H^n(Y;\uuline{\mathcal{S}_0})$ is specified by the following two canonical short exact sequences:
\begin{enumerate}
\item
$0 \to K_0 \to  H^n(Y;\uuline{\mathcal{S}_0}) \to H^n(Y^o) \to 0$
\item
$0 \to H^n_c(Y^o) \to H^n(Y;\uuline{\mathcal{S}_0}) \to C_0 \to 0$
\end{enumerate}
where $K_0=Im(H^n_c(c^o L) \to H^n_c(Y^o))$, $C_0=Im(H^n(Y^o) \to H^{n+1}_c(c^o L))$
with the non singular part $Y^o$ and $L = S^2 \times S^3$ the link of the singular point.  In addition,
$H^n(Y;\uuline{\mathbb{Q}}) \simeq H^n_c(Y^o;\uuline{\mathbb{Q}})$, $H_c^n(c^o L;\uuline{\mathbb{Q}}) \simeq H^n(L;\uuline{\mathbb{Q}})$
and $H^n(\widetilde{Y};\uuline{\mathbb{Q}}) \simeq H^n(Y;\uuline{IC})$.
\begin{table}[h]
\caption{Dimension of $H^n(Y;*)$}
\centering
\begin{tabular}{|c|c|c|c|} 
\hline
Deg & $H^n(Y;\uuline{\mathbb{S}_0})$ & $H^n(Y;\uuline{\mathbb{Q}})$ & $H^n(\widetilde{Y};\uuline{\mathbb{Q}})$\\
\hline
6 & 1 & 1 & 1 \\
\hline
5 & 0 & 0  & 0 \\
\hline
4 & 1 & 2 & 2  \\
\hline
3 & 204 & 203 & 202  \\
\hline
2 & 1 & 1 & 2 \\
\hline
1 & 0 & 0 & 0 \\
\hline
0 & 1 & 1  & 1 \\
\hline
\end{tabular}
\label{table:coho}
\end{table}
\quash{
\begin{table}[h]
\caption{Dimension of $H^n(Y;*)$ for $\uuline{\mathbb{S}_0}$, $\uuline{IC}$, and $\uuline{\mathbb{Q}}$.}
\centering
\begin{tabular}{|c|c|c|c||c|c|c||c|c|c|} 
\hline
Deg & $H^n(Y;\uuline{\mathbb{S}_0})$ & $H^n(Y;\uuline{\mathbb{Q}})$ & $H^n(\widetilde{Y};\uuline{\mathbb{Q}})$ 
& $H_c^n(c^o L;\uuline{\mathbb{Q}})$ & $H^n_c(Y^o;\uuline{\mathbb{Q}})$ & $K_0$ 
&  $H^n(Y^o;\uuline{\mathbb{Q}})$ &  $H^{n+1}_c(c^o L;\uuline{\mathbb{Q}})$ & $C_0$ \\
\hline
6 & 1 & 1 & 1 & 0 & 1 & 0 & 1 & 0 & 0\\
\hline
5 & 0 & 0  & 0 & 1 & 0 & 1 & 0 & 0 & 0\\
\hline
4 & 1 & 2 & 2  & 0 & 2 & 0 & 1 & 1 & 1\\
\hline
3 & 204 & 203 & 202  & 1 & 203 & 1 & 202 & 0 & 0\\
\hline
2 & 1 & 1 & 2  & 1 & 1 & 1 & 1 & 1 & 1\\
\hline
1 & 0 & 0 & 0  & 0 & 0 & 0 & 0 & 1 & 0\\
\hline
0 & 1 & 1  & 1 & 1 & 1 & 1 & 1 & 0 & 0\\
\hline
\end{tabular}
\label{table:coho}
\end{table}
\begin{table}[h]
\caption{Validation Check of Theorem \ref{mainthm}}
\centering
\begin{tabular}{|c|c|c|c||c|c|c|} 
\hline
Deg & $K_0$ & $H^n(Y;\uuline{\mathbb{S}_0})$ & $H^n(Y^o;\uuline{\mathbb{Q}})$ 
& $H_c^n(Y^o ;\uuline{\mathbb{Q}})$ & $H^n(Y;\uuline{\mathbb{S}_0})$ &  $C_0$ \\
\hline
6 & 0 & 1 & 1 & 1 & 1 & 0  \\
\hline
5 & 1 & 0  & 0 & 0 & 0 & 0 \\
\hline
4 & 0 & 1 & 1  & 2 & 1 & 1  \\
\hline
3 & 1 & 204 & 202  & 203 & 204 & 0  \\
\hline
2 & 1 & 1 & 1  & 1 & 1 & 1  \\
\hline
1 & 0 & 0 & 0  & 0 & 0 & 0 \\
\hline
0 & 1 & 1  & 1 & 1 & 1 & 0  \\
\hline
\end{tabular}
\label{table:coho2}
\end{table}
}
\subsection{String Theoretic Description}
The analysis of Type~IIB superstring compactification on $Y$ \`a la Strominger \cite{stro:massless} would proceed as a
 limiting process, starting from some $Y_{a_5\neq0}$, and taking the limit $a_5\to0$. For $|a_5|$ sufficiently small,
 $Y_{a_5\neq0}$ is a smooth quintic and a Type~IIB superstring compactification would feature one massless vector
 supermultiplet of $N=2$ supersymmetry in the effective 4-dimensional spacetime, and 101 massless hypermultiplets plus
 their Hermitian conjugates. In the limit $a_5\to0$ one of these hypermultiplets and its Hermitian conjugate would be
 ``frozen'' to a constant value and  become removed from the spectrum of massless (variable) fields. However,
 Strominger showed that a massless state and its Hermitian conjugate would become massless in the $a_5\to0$ limit where
 $|a_5|$ is small enough, $Y_{a_5 \neq 0}$ contains a ``vanishing'' $S^3$, which shrinks to a point in the
 $a_5 \to 0$ limit.  Strominger's replacement state turns out to have a mass proportional to the volume of the vanishing
 $S^3$. This is massive in compactifications on $Y_{a_5\neq0}$, but becomes massless in compactifications on $Y=Y_0$.

Therefore, the spectrum of massless fields in Type~IIB superstring compactifications on $Y_{a_5}$ remains constant in
 the $a_5 \to 0$ limit which can be obtained from Table~\ref{table:coho}:
\begin{equation}
\dim H^n(Y;\uuline{\mathbb{S}_0}) = \dim H^n(Y_{a_5\neq0};\uuline{\mathbb{Q}}).
 \label{e:EQ}
\end{equation}
Note that in singular 3-folds with more than one node that \textit{do} have a K\"ahler small resolution, the
$\uuline{\mathbb{S}_0}$-valued cohomology of the singular model would equal neither the $\uuline{\mathbb{Q}}$-valued
 cohomology of the small resolution, nor that of the smooth deformation. The equality~(\ref{e:EQ}) owes to the fact
 that $Y$ has a single node, whereupon its small resolution cannot be K\"ahler.
%:TH: End of changes

\section{Final Remarks}
In this paper we presented the construction and properties of a self-dual perverse sheaf $\uuline{\mathcal{S}_0}$ using 
techniques of MacPherson-Vilonen ~\cite{mv:elem}.  In addition, this
perverse sheaf was shown to satisfy Poincar\`e duality (property (3) of the K\"ahler package) as well as having relevant
applications for cases of Type IIB superstring compactification.  It is currently 
not known whether the remainder of the K\"ahler package holds for $\uuline{\mathcal{S}_0}$.  These are currently open problems.  
\subsection{The K\"ahler Package}
Proofs of all four parts of the K\"ahler package were not addressed in this paper.
Only Corollary \ref{poindual}, which was a consequence of the self-duality of $\uuline{\mathcal{S}_0}$,
provided Poincare duality.  It would take additional effort to prove the remaining parts of the
K\"ahler package.  The difficulties would lie in proving Hodge decomposition and
the Kunneth formula. It is not clear if there is a pure or mixed Hodge structure of $\uuline{\mathcal{S}_0}$.  This has
to be explored.   

\subsection{Multiple Singular Points}
There are cases that arise in String theory where the target space will have more than one singularity. 
We would like to extend this effort (one singular point) to include the case of multiple singular points.
Although this may seem to be a straight forward transition, there appear
to be problems in preserving some of the properties of the needed maps.  

In the case of multiple singular points, the singularity $\{y\}$ now becomes a singular set made up of isolated distinct singular points
$\Sigma=\{y_1,...,y_r\}$ where r is the number of singular points.  As a result, the injection $j$
now becomes  $\hat{j}: \Sigma \hookrightarrow Y$ while $i:Y^o \to Y$ remains unchanged. 
Consequently, the maps in the long exact 
sequences that arise from the distinguished triangles are altered by the inclusion of more singular points.  In
order to qualify this difficulty we can begin by restating the last three identities of Lemma \ref{cohvalues} as follows.

\begin{lem} \label{cohvalues2}
Let $Y$ be a simple stratified space and let $\uuline{\mathcal{F}}^\cdot$ be a complex of sheaves on Y.  
Let $\Sigma=\{y_1,...,y_r\} \in Y$ be the singular set and let $U_{y_b}$ be a distinguished neighborhood of $y_b$ so that
$\bar{U}_{y_b} \simeq cL_{y_b}$ and $U_{y_b} \simeq c^oL_{y_b}$ for $1 \leq b \leq r$.  Let $i:Y^o \hookrightarrow Y$
and $\hat{j}: \Sigma \hookrightarrow Y$ be inclusions.  Then there are natural isomorphisms
\begin{enumerate}
\item
$H^m(Y;\hat{j}_*\hat{j}^*\uuline{\mathcal{F}}^\cdot) \cong \bigoplus_{b=1}^r H^m(U_{y_b};\uuline{\mathcal{F}}^\cdot)$, $\forall $ $m > 0$,
\item
$H^m(Y;\hat{j}_*\hat{j}^!\uuline{\mathcal{F}}^\cdot) \cong \bigoplus_{b=1}^r H^m_c(U_{y_b};\uuline{\mathcal{F}}^\cdot)$, $\forall $ $m > 0$,
\item 
$H^m(Y;\hat{j}^*i_*\uuline{\mathcal{F}}^\cdot) \cong \bigoplus_{b=1}^r H^m(L_{y_b} \times (0,1);\uuline{\mathcal{F}}^\cdot)$, $\forall $ 
$m > 0$ where $L_{y_b}$ is the link of the singular point $y_b$ for $1 \leq b \leq r$.
\end{enumerate}
\end{lem}
\noindent
(Remark: Note that $U_{y_b} \cong cL_{y_b}$ where $cL_{y_b}$ is the cone over the link of the singular point $y_b$.)\\
\noindent
Consider the distinguished triangles in diagrams (\ref{d1}) and (\ref{d2}) applied to $\uuline{\mathcal{S}_0}^\cdot$.  Rewriting these triangles using 
$\hat{j}:\Sigma \hookrightarrow Y$ instead of $j:\{y\} \hookrightarrow Y$ we obtain,
\begin{equation}\label{d1jhat} 
\xymatrix{
R\hat{j}_*\hat{j}^!\uuline{\mathcal{S}_0}^\cdot \ar[rr]& &  \uuline{\mathcal{S}_0}^\cdot \ar[dl]\\ 
& Ri_*i^* \uuline{\mathcal{S}_0}^\cdot \ar[ul]^{[1]}\\ }
\end{equation}
\vspace{0.5cm}
\begin{equation} \label{d2jhat}
\xymatrix{
Ri_!i^* \uuline{\mathcal{S}_0}^\cdot \ar[rr]& &  \uuline{\mathcal{S}_0}^\cdot \ar[dl]\\ 
& R\hat{j}_*\hat{j}^* \uuline{\mathcal{S}_0}^\cdot \ar[ul]^{[1]}\\ }
\end{equation}
The resulting long exact sequence in degree n for the triangle in diagram (\ref{d1jhat}) gives,
\begin{equation}\label{d1les}
... \to \bigoplus_{b=1}^r H^n_c( cL_{y_b} ; \uuline{\mathcal{S}_0}^\cdot) \to H^n(Y;\uuline{\mathcal{S}_0}^\cdot) \to H^n(Y^o) \to \bigoplus_{b=1}^r H^{n+1}_c( cL_{y_b} ; \uuline{\mathcal{S}_0}^\cdot) \to  ...
\end{equation}
  
Let the maps of the long exact sequence in eqn. (\ref{d1les}) be defined as:
\begin{align}
\alpha_1: & \bigoplus_{b=1}^r H^n_c( cL_{y_b} ; \uuline{\mathcal{S}_0}^\cdot) \to H^n(Y;\uuline{\mathcal{S}_0}^\cdot)\\
\beta_1: & H^n(Y;\uuline{\mathcal{S}_0}^\cdot) \to H^n(Y^o) \\
\gamma_1: & H^n(Y^o) \to \bigoplus_{b=1}^r H^{n+1}_c( cL_{y_b} ; \uuline{\mathcal{S}_0}^\cdot)
\end{align}
Similarly the resulting long exact sequence in degree n for the triangle in diagram(\ref{d2jhat}) gives,
\begin{equation}\label{d2les}
... \to \bigoplus_{b=1}^r H^{n-1}( cL_{y_b} ; \uuline{\mathcal{S}_0}^\cdot) \to H^n_c(Y^o) \to H^n(Y;\uuline{\mathcal{S}_0}^\cdot) \to \bigoplus_{b=1}^r H^n( cL_{y_b} ; \uuline{\mathcal{S}_0}^\cdot) \to  ...
\end{equation}
Similarly let the maps of the long exact sequence in eqn. (\ref{d2les}) be defined as:
\begin{align}
\alpha_2: & \bigoplus_{b=1}^r H^{n-1}( cL_{y_b} ; \uuline{\mathcal{S}_0}^\cdot) \to H^n_c(Y^o)\\
\beta_2: & H^n_c(Y^o) \to H^n(Y;\uuline{\mathcal{S}_0}^\cdot) \\
\gamma_2: & H^n(Y;\uuline{\mathcal{S}_0}^\cdot) \to \bigoplus_{b=1}^r H^n( cL_{y_b} ; \uuline{\mathcal{S}_0}^\cdot)
\end{align}
  
Recall the figure where the map c is an injection and d is a surjection that was shown to exist
in the proof of Theorem \ref{mainthm}.  These were requirements imposed by String theory.  (See 
section 1.3 for a discussion.)
%and Proposition \ref{propsszero} for a formal statement.)  
\begin{equation} \label{d3new}
\xymatrix{
& & & &  H^n(Y;\mathcal{S}_0) \ar@{->>}[d]^d \\ 
  & H^{n-1}(i_*i^*\mathbb{Q})\ar[r]^\alpha & H^{n}(j_*j^!\mathbb{Q})\ar[r]^\beta \ar@{->>}[d]^a  
	&  H^n_c(Y^o) \ar@{^{(}->}[ur]^c \ar[r]^\gamma
    & H^{n}(Y^o) \ar[r]^\delta \ar@{->>}[d]^e & H^{n+1}(j_*j^!\mathbb{Q}) \\
	& & K_0=Im(\beta) \ar@{^{(}->}[ur]^b  & & C_0=Im(\delta) \ar@{^{(}->}[ur]^f }
\end{equation}
Equation (\ref{d3new}) and its maps were defined in the case of one singular point where the map $j:\{pt\} \to Y$.  Now if we replace
$j$ with $\hat{j}:\Sigma \to Y$ it is not completely clear if the map c will remain an injection and d will remain
a surjection.  (Recall that in order to obtain a cohomology for String theory it is a requirement that
the map c be injective and d be surjective.)  These are necessary conditions in order to obtain a larger rank of cohomology in the middle dimension.  Notice that in the long exact sequences using $\hat{j}$, eqns. 
(\ref{d1les}) and (\ref{d2les}),
if we replaces $j$ by $\hat{j}$, the maps c and d can be described as follows:
\begin{align}
c \text{ corresponds to the map } & \beta_2 \Rightarrow \text{ $\alpha_2$ needs to be the 0-map.} \\
d \text{ corresponds to the map } & \beta_1  \Rightarrow \text{ $\gamma_1$ needs to be the 0-map.}
\end{align}
  
It is not quite clear however, if for all singular points in $\Sigma$ the maps $\alpha_2$ and $\gamma_1$ can
be made to be the 0-maps.  This is currently an obstruction to extending the singular point case to the 
case of a singular set.  Hopefully a future effort will address this.
\\
\\
\noindent
\textbf{Acknowledgments}: The author wishes to extend special thanks to R.M Goresky and T. Hubsch for the very helpful
discussions and detailed reviews of the final manuscript.  The author also wishes to thank D. Massey for helpful 
discussions concerning perverse sheaves and a review of the initial manuscript.  Many thanks to R. MacPherson for
 making the initial remark that there was such a perverse sheaf that had the properties we were looking for. 
The Institute for Advanced Study,
 Princeton, NJ deserves many compliments for its hospitality and use of facilities during the development of this work.
\bigskip
\addcontentsline{toc}{section}{References}
%
%\bibliographystyle{amsplain}
%\bibliography{paper}

\begin{thebibliography}{text} \rightskip=0pt plus1fill

\bibitem{hub:stringy}
	Tristan Hubsch: {\itshape  On A Stringy Singular Cohomology}, \MPL{A12}, 521-533 (1997). (hep-th/9612075)
	
\bibitem{mv:elem}
	Robert MacPherson and Kari Vilonen: {\itshape Elementary Constructions of Perverse Sheaves}, \IM{84} (1986) 403-435.

\bibitem{greenhub:conifold1}
	Paul S. Green and Tristan Hubsch: {\itshape Possible Phase Transitions Among Calabi-Yau Compactifications}, \PRL{61} 1163-1167 (1988).

\bibitem{greenhub:conifold2}
	Paul S. Green and Tristan Hubsch: {\itshape Connecting Moduli Spaces of Calabi-Yau Threefolds}, \CMP{119} 431-441 (1988).
	
\bibitem{stro:massless}
	Andrew Strominger: {\itshape Massless Black Holes and Conifolds in String Theory }, \NP{B451} 96-108 (1995). (hep-th/9504090)

\bibitem{hubrah:geomhom}
	Tristan Hubsch and Abdul Rahman: {\itshape On the Geometry and Homology of Certain Simple Stratified Varieties}, \JGP{53} 31-48 (2005). (math.AG/0210394)
	
\bibitem{hubsch:best}
	Tristan Hubsch: {\itshape Calabi-{Y}au Manifolds: A Bestiary for Physicists}, World Scientific Pub Co. (1992).

\bibitem{hub:singspace}
	Tristan Hubsch: {\itshape How Singular a Space Can Superstrings Thread?}, \MPL{A6} 207-216 (1991).
	
\bibitem{gormac:ih2}
	Mark Goresky and Robert MacPherson: {\itshape Intersection Homology II}, \IM{71} 77-129 (1983).

\bibitem{cheegormac:L2cohom}
	J. Cheeger and Mark Goreksy and Robert MacPherson: {\itshape L2-Cohomology and Intersection 
	Homology of Singular Algebraic Varieties}, Seminar on Differential Geometry, S.T. Yau ed. \AMS{102} 303-340 (1982).

\bibitem{gormac:ih1}
	Mark Goresky and Robert MacPherson: {\itshape Intersection Homology Theory}, Topology \textbf{19} 135-162 (1980). 

\bibitem{kassch:shvmfd}
	M. Kashiwara and P. Schapira: {\itshape Sheaves on Manifolds}, Springer-Verlag (1980).

\bibitem{dimca:shv}
	Alexandru Dimca: {\itshape Sheaves in Topology}, Springer (2004).

\bibitem{gelfman:homalg}
	S. I. Gelfand and Yuri I. Manin: {\itshape Homological Algebra}, Springer (1999).
	
\bibitem{verd:dualcoh}
	J.L. Verdier:{\itshape Dualit\`e Dans La Cohomologie des Espaces Localment Compactes}, Seminar Bourbaki \textbf{300} (1965).

\bibitem{iver:cohsheav}
	Birger Iversen: {\itshape Cohomology of Sheaves}, Aarhus universitet, Matematisk Institut (1984).

\bibitem{gormac:wghtdcoh}
	R.M. Goresky, G. Harder and R. MacPherson: {\itshape Weighted Cohomology}, \IM{116} 139-213 (1994). 
	
\end{thebibliography}
%\nocite{*}
%

\end{document}